\documentclass[12pt, reqno]{amsart}
\usepackage{amsmath, amsthm, amscd, amsfonts, amssymb, graphicx, color}
\usepackage[bookmarksnumbered, colorlinks, plainpages]{hyperref}

\hypersetup{colorlinks=true,linkcolor=red, anchorcolor=green, citecolor=cyan, urlcolor=red, filecolor=magenta, pdftoolbar=true}
\newtheorem{theorem}{Theorem}[section]
\newtheorem{lemma}[theorem]{Lemma}

\theoremstyle{definition}

\theoremstyle{remark}

\numberwithin{equation}{section}

\begin{document}

\setcounter{page}{1}


\begin{center}
{\Large \textbf{Approximation properties of the Stancu type Dunkl generalization of the Kantorovich-Sz\'{a}sz-Mirakjan-operators via $q$-calculus}}

\bigskip

\textbf{M. Mursaleen}, \textbf{Taqseer Khan} and \textbf{Nasiruzzaman}

Department of\ Mathematics, Aligarh Muslim University, Aligarh--202002, India%
\\[0pt]

mursaleenm@gmail.com; nasir3489@gmail.com \\[0pt]

\bigskip

\textbf{Abstract}
\end{center}

\parindent=8mm {\footnotesize {In this paper we construct Stancu type $q$-Kantrovich-Sz\'{a}sz-Mirakjan operators generated by Dunkl generalization of the
exponential function. We obtain some approximation results
using the Korovkin approximation theorem and the weighted Korovkin-type theorem
for these operators. We also study convergence properties by using the
modulus of continuity and the rate of convergence of these operators for
functions belonging to the Lipschitz class. Furthermore, we obtain the rate
of convergence in terms of the classical, second order, and weighted modulus
of continuity.}}\newline












\section{Introduction and preliminaries}

In 1912, S.N Bernstein \cite{sbbl1} introduced the following sequence of
operators $B_{n}:C[0,1]\rightarrow C[0,1]$ defined by
\begin{equation}
B_{n}(f;x)=\sum_{k=0}^{n}\binom{n}{k}x^{k}(1-x)^{n-k}f\left( \frac{k}{n}%
\right) ,~~~~~~~x\in \lbrack 0,1].  \label{s1}
\end{equation}%
for $n\in \mathbb{N}$ and $f\in C[0,1]$.

In 1950, for $x \geq 0$, Sz\'{a}sz \cite{sbbl4} introduced the operators
\begin{equation}  \label{s2}
S_n(f;x)=e^{-nx}\sum_{k=0}^\infty \frac{(nx)^k}{k!} f\left(\frac{k}{n}%
\right),~~~~~~~f \in C[0,\infty).
\end{equation}

In the field of approximation theory, the application of $q$-calculus
emerged as a new area in the field of approximation theory. The first $q$%
-analogue of the well-known Bernstein polynomials was introduced by Lupa\c{s}
by applying the idea of $q$-integers \cite{sbbl2}. In 1997 Phillips \cite%
{sbbl3} considered another $q$-analogue of the classical Bernstein
polynomials. Later on, many authors introduced $q$-generalizations of
various operators and investigated several approximation properties \cite%
{skant, smah1,sradu,smah2,smur3}.\newline


The $q$-integer $[n]_{q}$, the $q$-factorial $[n]_{q}!$ and the $q$-binomial
coefficient are defined by (see \cite{sbbl5})
\begin{align*}
\lbrack n]_{q}& :=\left\{
\begin{array}{ll}
\frac{1-q^{n}}{1-q}, & \hbox{if~}q\in \mathbb{R}^{+}\setminus \{1\} \\
n, & \hbox{if~}q=1,%
\end{array}%
\right. \mbox{for $n\in \mathbb{N} $~and~$[0]_q=0$}, \\
\lbrack n]_{q}!& :=\left\{
\begin{array}{ll}
\lbrack n]_{q}[n-1]_{q}\cdots \lbrack 1]_{q}, & \hbox{$n\geq 1$,} \\
1, & \hbox{$n=0$,}%
\end{array}%
\right. \\
\left[
\begin{array}{c}
n \\
k%
\end{array}%
\right] _{q}& :=\frac{[n]_{q}!}{[k]_{q}![n-k]_{q}!},
\end{align*}%
respectively.

\noindent The $q$-analogue of $(1+x)^{n}$ is the polynomial
\begin{equation*}
(1+x)_{q}^{n}:=\left\{
\begin{array}{ll}
(1+x)(1+qx)\cdots (1+q^{n-1}x) & \quad n=1,2,3,\cdots \\
1 & \quad n=0.%
\end{array}%
\right.
\end{equation*}%
A $q$-analogue of the common Pochhammer symbol also called a $q$-shifted
factorial is defined as
\begin{equation*}
(x;q)_{0}=1,~(x;q)_{n}=\prod\limits_{j=0}^{n-1}(1-q^{j}x),~(x;q)_{\infty
}=\prod\limits_{j=0}^{\infty }(1-q^{j}x).
\end{equation*}

\noindent The Gauss binomial formula is given by
\begin{equation*}
(x+a)_{q}^{n}=\sum\limits_{k=0}^{n}\left[
\begin{array}{c}
n \\
k%
\end{array}%
\right] _{q}q^{k(k-1)/2}a^{k}x^{n-k}.
\end{equation*}

The $q-$analogue of Bernstein operators \cite{sbbl3} is defined as

\begin{equation}
B_{n,q}(f;x)=\sum\limits_{k=0}^{n}\left[
\begin{array}{c}
n \\
k%
\end{array}%
\right] _{q}x^{k}\prod\limits_{s=0}^{n-k-1}(1-q^{s}x)~~f\left( \frac{[k]_{q}
}{[m]_{q}}\right) ,~~x\in [0,1], n \in \mathbb{N}.
\end{equation}

There are two $q$-analogue of the exponential function $e^{z}$, defined as
follows :\newline

For $\mid z \mid < \frac{1}{1-q}$ and $\mid q \mid <1$,
\begin{equation}  \label{zfr1}
e(z)= \sum_{k=0}^\infty \frac{z^k}{k!}= \frac{1}{1-\left((1-q)z\right)_q^%
\infty},
\end{equation}
and for $\mid q \mid<1$,
\begin{equation}  \label{zfr2}
E(z)= \prod_{j=0}^\infty \left(1+(1-q)q^jz\right)_q^\infty=\sum_{k=0}^\infty
q^{\frac{k(k-1)}{2}}\frac{z^k}{k!} = \left(1+(1-q)z\right)_q^\infty,
\end{equation}
where $(1-x)_q^\infty=\prod_{j=0}^\infty(1-q^jx)$.

Sucu \cite{sbbl9} defined a Dunkl analogue of Sz\'{a}sz operators via a
generalization of the exponential function \cite{sbbl10} as follows:

\begin{equation}
S_{n}^*(f;x):= \frac{1}{e_\mu(nx)}\sum_{k=0}^\infty \frac{(nx)^k}{%
\gamma_\mu(k)} f \left(\frac{k+2\mu\theta_k}{n}\right),
\end{equation}
where $x \geq 0,~~~f \in C[0,\infty),\mu \geq 0,~~~n \in \mathbb{N}$\newline
and
\begin{equation*}
e_\mu(x)= \sum_{n=0}^\infty \frac{x^n}{\gamma_\mu(n)}.
\end{equation*}
Here
\begin{equation*}
\gamma_\mu(2k)= \frac{2^{2k}k!\Gamma\left(k+\mu+\frac{1}{2}\right)}{%
\Gamma\left(\mu+\frac{1}{2}\right)},
\end{equation*}
and
\begin{equation*}
\gamma_\mu(2k+1)= \frac{2^{2k+1}k!\Gamma\left(k+\mu+\frac{3}{2}\right)}{%
\Gamma\left(\mu+\frac{1}{2}\right)}.
\end{equation*}

There is given a recursion for $\gamma _{\mu }$
\begin{equation*}
\gamma _{\mu }(k+1)=(k+1+2\mu \theta _{k+1})\gamma _{\mu
}(k),~~~~k=0,1,2,\cdots ,
\end{equation*}%
where
\begin{equation*}
\theta _{k}=%
\begin{cases}
0 & \quad \text{if }k\in 2\mathbb{N} \\
1 & \quad \text{if }k\in 2\mathbb{N}+1.%
\end{cases}%
\end{equation*}

Cheikh et al. \cite{sbbl11} stated the $q$-Dunkl classical $q$-Hermite type
polynomials and gave definitions of $q$-Dunkl analogues of exponential
functions and recursion relations for $\mu >-\frac{1}{2}$ and $0<q<1$.
\begin{equation}
e_{\mu ,q}(x)=\sum_{n=0}^{\infty }\frac{x^{n}}{\gamma _{\mu ,q}(n)},~~~x\in
\lbrack 0,\infty )  \label{sr1}
\end{equation}%
\begin{equation}
E_{\mu ,q}(x)=\sum_{n=0}^{\infty }\frac{q^{\frac{n(n-1)}{2}}x^{n}}{\gamma
_{\mu ,q}(n)},~~~x\in \lbrack 0,\infty )  \label{sr2}
\end{equation}%
\begin{equation}
\gamma _{\mu ,q}(n+1)=\left( \frac{1-q^{2\mu \theta _{n+1}+n+1}}{1-q}\right)
\gamma _{\mu ,q}(n),~~~~~n\in \mathbb{N},  \label{sr3}
\end{equation}

\begin{equation*}
\theta _{n}=%
\begin{cases}
0 & \quad \text{if }n\in 2\mathbb{N}, \\
1 & \quad \text{if }n\in 2\mathbb{N}+1.%
\end{cases}%
\end{equation*}%
An explicit formula for $\gamma _{\mu ,q}(n)$ is
\begin{equation*}
\gamma _{\mu ,q}(n)=\frac{(q^{2\mu +1},q^{2})_{[\frac{n+1}{2}%
]}(q^{2},q^{2})_{[\frac{n}{2}]}}{(1-q)^{n}}\gamma _{\mu ,q}(n),~~~~~n\in
\mathbb{N}.
\end{equation*}

And some of the special cases of $\gamma_{\mu,q}(n)$ defined as:

\begin{equation*}
\gamma_{\mu,q}(0)=1, ~~~~~~\gamma_{\mu,q}(1)=\frac{1-q^{2\mu+1}}{1-q}%
,~~~~\gamma_{\mu,q}(2)=\left(\frac{1-q^{2\mu+1}}{1-q}\right)\left(\frac{1-q^2%
}{1-q}\right),
\end{equation*}

\begin{equation*}
\gamma_{\mu,q}(3)= \left(\frac{1-q^{2\mu+1}}{1-q}\right) \left(\frac{1-q^2}{%
1-q}\right) \left(\frac{1-q^{2\mu+3}}{1-q}\right),
\end{equation*}

\begin{equation*}
\gamma_{\mu,q}(4)= \left(\frac{1-q^{2\mu+1}}{1-q}\right) \left(\frac{1-q^2}{%
1-q}\right) \left(\frac{1-q^{2\mu+3}}{1-q}\right) \left(\frac{1-q^4}{1-q}%
\right).
\end{equation*}

In \cite{ckz}, G\"{u}rhan I\c{c}\"{o}z gave a Dunkl generalization of
Kantrovich type integral generalization of Sz\'{a}sz operators. In \cite%
{sbbl13}, they gave a Dunkl generalization of Sz\'{a}sz operators via $q$%
-calculus as:
\begin{equation}
D_{n,q}(f;x)=\frac{1}{e_{\mu ,q}([n]_{q}x)}\sum_{k=0}^{\infty }\frac{%
([n]_{q}x)^{k}}{\gamma _{\mu ,q}(k)}f\left( \frac{1-q^{2\mu \theta _{k}+k}}{%
1-q^{n}}\right) ,  \label{sss1}
\end{equation}%
for $\mu >\frac{1}{2},~~x\geq 0,~~0<q<1$ and $f\in C[0,\infty ).$

\begin{lemma}
\label{vd}

\begin{enumerate}
\item \label{vd1} $D_{n,q}(1;x)=1$,

\item \label{vd2} $D_{n,q}(t;x)=x$,

\item \label{vd3}$x^2+[1-2 \mu]_q q^{2 \mu}\frac{e_{\mu,q}(q [n]_q(x))}{%
e_{\mu,q}([n]_q x)}\frac{x}{[n]_q} \leq D_{n,q}(t^2;x) \leq x^2+[1+2 \mu]_q \frac{%
x}{[n]_q}$,

\item \label{vd4}$D_{n,q}(t^3;x) \geq x^3+(2q+1)[1-2\mu]_q \frac{%
e_{\mu,q}(q[n]_qx)}{e_{\mu,q}([n]_qx)}\frac{x^2}{[n]_q} +q^{4
\mu}[1-2\mu]_q^2 \frac{e_{\mu,q}(q^2[n]_qx)}{e_{\mu,q}([n]_qx)}\frac{x}{%
[n]_q^2}$,

\item \label{vd5}$D_{n,q}(t^4;x) \leq x^4+ 6[1+2\mu]_q \frac{x^3}{[n]_q}
+7[1+2\mu]_q^2\frac{x^2}{[n]_q^2} +[1+2\mu]_q^3 \frac{x}{[n]_q^3}$.
\end{enumerate}
\end{lemma}

In this paper we construct Kantrovich type Sz\'{a}sz-Mirakjan operators
generated by Dunkl generalization of the exponential function via $q$%
-integers. We obtain some approximation results via well known Korovkin's
type theorem and weighted Korovkin's type theorem for these operators. We
also study convergence properties by using the modulus of continuity and the
rate of convergence of the operators for functions belonging to the
Lipschitz class. Furthermore, we obtain the rate of convergence in terms of
the classical, second order, and weighted modulus of continuity.

\section{auxiliary results}

We define a Dunkl generalization of Sz\'{a}sz-Mirakjan-Kantrovich operators via $q$-calculs as follows: \newline

For any $x\in [0,\infty), ~~~n \in \mathbb{N}, ~~0<q< 1,$ and $\mu>\frac{1}{2%
}$, we define

\begin{equation}
T_{n,q}^{\ast}(f;x)=\frac{[n]_{q}}{e_{\mu ,q}([n]_{q}x)}\sum_{k=0}^{\infty }%
\frac{([n]_{q}x)^{k}}{\gamma _{\mu ,q}(k)}\int_{\frac{q[k+2\mu \theta
_{k}]_{q}}{[n]_{q}}}^{\frac{[k+1+2\mu \theta _{k}]_{q}}{[n]_{q}}}f\left(\frac{nt+\alpha}{n+\beta}\right)d_{q}t,
\label{snss1}
\end{equation}%
where $f$ is a continuous and nondecreasing function on $[0,\infty )$.

\begin{lemma}
\label{snlm1} Let $T_{n,q}^{\ast}(.~;~.)$ be the operators given by %
\eqref{snss1}. Then we have the following identities and inequalities:

\begin{enumerate}
\item \label{snlm11} $T_{n,q}^{\ast}(1;x)=1$,

\item \label{snlm12} $T_{n,q}^{\ast}(t;x)=\frac{2qn}{(n+\beta)[2]_q}x+\frac{n}{(n+\beta)[2]_q[n]_q}+\frac{\alpha}{n+\beta}$,

\item \label{snlm13}$\frac{n^2}{(n+\beta)^2}\frac{1}{[3]_q[n]_q^2}+\frac{2n\alpha}{(n+\beta)^2[2]_q[n]_q}+\frac{\alpha^2}{(n+\beta)^2}
    +\Big(\frac{n^2}{(n+\beta)^2}\frac{3q}{[3]_q[n]_q}+\frac{2n\alpha}{(n+\beta)^2}\frac{2q}{[2]_q}+\frac{3q^{2(\mu+1)}}{[3]_q[n]_q}\\
\times[1-2\mu]_q \frac{e_{\mu,q}(q [n]_q(x))}{e_{\mu,q}([n]_q x)}%
\Big)x+\frac{n^2}{(n+\beta)^2}\frac{3q^2}{[3]_q}x^2 \leq T_{n,q}^{\ast}(t^2;x) \leq \frac{n^2}{(n+\beta)^2}\frac{1}{[3]_q[n]_q^2}+\frac{2n\alpha}{(n+\beta)^2[2]_q[n]_q}+\frac{\alpha^2}{(n+\beta)^2}
+\Big(\frac{n^2}{(n+\beta)^2}\frac{3}{[3]_q[n]_q}+\frac{2n\alpha}{(n+\beta)^2}\frac{2}{[2]_q}
+\frac{3q^{2(\mu+1)}}{[3]_q[n]_q}[1-2\mu]_q \frac{e_{\mu,q}(q [n]_q(x))}{e_{\mu,q}([n]_q x)}\Big)x\\
+\frac{n^2}{(n+\beta)^2}\frac{3}{[3]_q}x^2 $,

\item \label{snlm14}$\frac{n}{(n+\beta)^3[n]_q}\Big(\frac{n^2}{[4]_q[n]_q^2}+\frac{3n\alpha}{[3]_q[n]_q}+\frac{3\alpha^2}{[2]_q}\Big)+\frac{\alpha^3}{(n+\beta)^3}+
    \frac{1}{(n+\beta)^3}\{n\big( \frac{4n^2q}{[4]_q[n]_q^2}+\frac{3n\alpha}{[n]_q}\frac{3q}{[3]_q}+\frac{6\alpha^2q}{[2]_q}\big)+(\frac{2n}{[4]_q[n]_q}
    +\frac{3\alpha}{[3]_q})\frac{3n^2q^{2(\mu+1)}}{[n]_q}[1-2\mu]_q\frac{e_{\mu,q}(q [n]_q(x))}{e_{\mu,q}([n]_q x)}+\frac{4n^3}{[4]_q[n]_q^2}[1-2\mu]_q^2
    \frac{e_{\mu,q}(q^2[n]_q(x))}{e_{\mu,q}([n]_q(x))}\}x+
\frac{n^2}{(n+\beta)^3}\Big(\frac{6nq^2}{[4]_q[n_q]}+\frac{9\alpha q^2}{[3]_q}+\frac{4nq^3}{[4]_q[n]_q}(2q+1)[1-2\mu]_q\frac{e_{\mu,q}(q[n]_qx)}{e_{\mu,q}([n]_q x)}\Big)x^2
+\frac{4n^3q^3}{(n+\beta)^3[4]_q}x^3 \leq
T_{n,q}^{\ast}(t^3;x) \leq \frac{n}{(n+\beta)^3[n]_q}\Big(\frac{n^2}{[4]_q[n]_q^2}+\frac{3n\alpha}{[3]_q[n]_q}+\frac{3\alpha^2}{[2]_q}\Big)+\frac{\alpha^3}{(n+\beta)^3}
+\frac{1}{(n+\beta)^3}(\frac{4n^3}{[4]_q[n]_q^2}(1+[1+2\mu]_q^2)+\frac{9n^2\alpha}{[3]_q[n]_q}(1+[1+2\mu]_q)+6n(\frac{\alpha^2}{[2]_q}
+\frac{n^2[1+2\mu]_q}{[4]_q[n]_q}))x
+\frac{1}{(n+\beta)^3}\{\frac{6n^3}{[4]_q[n]_q}(1+2[1+2\mu]_q)+\frac{9n^2\alpha}{[3]_q}\}x^2
+\frac{4n^3}{(n+\beta)^3[4]_q}x^3$,

\item \label{snlm15}$T_{n,q}^{\ast}(t^4;x)\leq \frac{n}{(n+\beta)^4[n]_q}(\frac{n^3}{[5]_q[n]_q^3}+\frac{4n^2\alpha}{[4]_q[n]_q^2}+\frac{6n\alpha^2}{[3]_q[n]_q}+\frac{4\alpha^3}{[2]_q})
+\frac{\alpha^4}{(n+\beta)^4}
+\{\frac{n^2}{(n+\beta)^4}(\frac{5n^2}{[5]_q[n]_q^3}+\frac{16n\alpha}{[4]_q[n]_q^2}+\frac{18\alpha^2}{[3]_q[n]_q}+\frac{8n\alpha}{[2]_q})
+\frac{2n^2}{(n+\beta)^4}(\frac{5n^2}{[5]_q[n]_q^2}+\frac{12n\alpha}{[4]_q[n]_q}+\frac{9\alpha^2}{[3]_q})\frac{[1+2\mu]_q}{[n]_q}
+\frac{2n^3}{(n+\beta)^4}(\frac{5n}{[5]_q[n]_q}\frac{8\alpha}{[n]_q}+)\\
\times\frac{[1+2\mu]_q^2}{[n]_q^2}+\frac{5n^4}{(n+\beta)^4[5]_q}\frac{[1+2\mu]_q^3}{[n]_q^3}\}x
+\{\frac{2n^2}{(n+\beta)^4}(\frac{5n^2}{[5]_q[n]_q^2}+\frac{12n\alpha}{[4]_q[n]_q}+\frac{9\alpha^2}{[3]_q})
+\frac{6n^3}{(n+\beta)^4}(\frac{5n}{[5]_q[n]_q}+\frac{8\alpha}{[n]_q})\frac{[1+2\mu]_q}{[n]_q}
+\frac{35n^4}{(n+\beta)^4[5]_q}\frac{[1+2\mu]_q^2}{[n]_q^2}\}x^2
+\{\frac{2n^3}{(n+\beta)^4}(\frac{5n}{[5]_q[n]_q}+\frac{8\alpha}{[n]_q})+\frac{5n^4}{(n+\beta)^4[5]_q}\frac{[1+2\mu]_q}{[n]_q}\}x^3
+\frac{5n^4}{(n+\beta)^4[5]_q}x^4$.
\end{enumerate}
\end{lemma}

\begin{proof}
It is easily seen that
\begin{equation}
\lbrack k+1+2\mu \theta _{k}]_{q}=q[k+2\mu \theta _{k}]_{q}+1.  \label{snmn1}
\end{equation}

so we get the followings
\begin{equation}  \label{snmn2}
\int_{\frac{q[k+2\mu \theta_k]_q}{[n]_q}}^{\frac{[k+1+2\mu \theta_k]_q}{[n]_q}%
} 1 d_q t=\frac{1}{[n]_q},
\end{equation}

\begin{equation}  \label{snmn3}
\int_{\frac{q[k+2\mu \theta_k]_q}{[n]_q}}^{\frac{[k+1+2\mu \theta_k]_q}{[n]_q}%
} t d_q t =\frac{1}{[2]_q[n]_q^2}\left(1+2q[k+2\mu \theta_k]_q\right),
\end{equation}

\begin{equation}  \label{snmn4}
\int_{\frac{q[k+2\mu \theta_k]_q}{[n]_q}}^{\frac{[k+1+2\mu \theta_k]_q}{[n]_q}%
} t^2 d_q t =\frac{1}{[3]_q[n]_q^3}\left(1+3q[k+2\mu \theta_k]_q+3q^2[k+2\mu
\theta_k]_q^2 \right),
\end{equation}

\begin{equation}
\int_{\frac{q\lbrack k+2\mu \theta _{k}]_{q}}{[n]_{q}}}^{\frac{[k+1+2\mu
\theta _{k}]_{q}}{[n]_{q}}}t^{3}d_{q}t=\frac{1}{[4]_{q}[n]_{q}^{4}}\left(
1+4q[k+2\mu \theta _{k}]_{q}+6q^{2}[k+2\mu \theta
_{k}]_{q}^{2}+4q^{3}[k+2\mu \theta _{k}]_{q}^{3}\right) ,  \label{snmn242}
\end{equation}%
and
\begin{equation}
\int_{\frac{q\lbrack k+2\mu \theta _{k}]_{q}}{[n]_{q}}}^{\frac{[k+1+2\mu
\theta _{k}]_{q}}{[n]_{q}}}t^{4}d_{q}t=\frac{1}{[5]_{q}[n]_{q}^{5}}\left(
1+5q[k+2\mu \theta _{k}]_{q}+10q^{2}[k+2\mu \theta
_{k}]_{q}^{2}+10q^{3}[k+2\mu \theta _{k}]_{q}^{3}+5q^{4}[k+2\mu \theta
_{k}]_{q}^{4}\right) .  \label{snmn24}
\end{equation}%
From the Lemma \ref{vd}
we have the following results:
\begin{equation}
\frac{1}{[n]_{q}}\frac{1}{e_{\mu ,q}([n]_{q}x)}\sum_{k=0}^{\infty }\frac{%
([n]_{q}x)^{k}}{\gamma _{\mu ,q}(k)}[k+2\mu \theta _{k}]_{q}=x,
\label{snb24}
\end{equation}%
\begin{equation}
x^{2}+q^{2\mu }[1-2\mu ]_{q}\frac{e_{\mu ,q}(q[n]_{q}x)}{e_{\mu ,q}([n]_{q}x)%
}\frac{x}{[n]_{q}}\leq \frac{1}{[n]_{q}^{2}}\frac{1}{e_{\mu ,q}([n]_{q}x)}%
\sum_{k=0}^{\infty }\frac{([n]_{q}x)^{k}}{\gamma _{\mu ,q}(k)}[k+2\mu \theta
_{k}]_{q}^{2}\leq x^{2}+[1+2\mu ]_{q}\frac{x}{[n]_{q}},  \label{snb25}
\end{equation}

\begin{equation}  \label{snub26}
\frac{1}{[n]_q^3}\frac{1}{e_{\mu,q}([n]_{q}x)}\sum_{k=0}^\infty \frac{%
([n]_{q}x)^k}{\gamma_{\mu,q}(k)}[k+2\mu \theta_k]_q^3 \leq x^3+3[1+2\mu]_q%
\frac{x^2}{[n]_q}+ [1+2\mu]_q^2\frac{x}{[n]_q^2},
\end{equation}

$\frac{1}{[n]_q^3}\frac{1}{e_{\mu,q}([n]_{q}x)}\sum_{k=0}^\infty \frac{%
([n]_{q}x)^k}{\gamma_{\mu,q}(k)}[k+2\mu \theta_k]_q^3$
\begin{equation}  \label{snb26}
\geq x^3+(2q+1)[1-2\mu]_q\frac{e_{\mu,q}(q[n]_qx)}{e_{\mu,q}([n]_qx)}\frac{%
x^2}{[n]_q}+ q^{4 \mu}[1-2\mu]_q^2\frac{e_{\mu,q}(q^2[n]_qx)}{%
e_{\mu,q}([n]_qx)}\frac{x}{[n]_q^2},
\end{equation}
and
\begin{equation}  \label{snb27}
\frac{1}{[n]_q^4}\frac{1}{e_{\mu,q}([n]_{q}x)}\sum_{k=0}^\infty \frac{%
([n]_{q}x)^k}{\gamma_{\mu,q}(k)}[k+2\mu \theta_k]_q^4 \leq x^4+[1+2\mu]_q%
\frac{x^3}{[n]_q}+7[1+2\mu]_q^2\frac{x^2}{[n]_q^2}+[1+2\mu]_q^3\frac{x}{%
[n]_q^3}.
\end{equation}

\begin{enumerate}
\item From \eqref{snmn2} we have $T_{n,q}^*(1;x)=\frac{[n]_q}{%
e_{\mu,q}([n]_{q}x)}\sum_{k=0}^\infty \frac{([n]_{q}x)^k}{\gamma_{\mu,q}(k)}%
\frac{1}{[n]_q}=1$. \newline

\item If $f(t)=t$ then \eqref{snss1}, \eqref{snmn3} and \eqref{snb24} imply
that
\begin{eqnarray*}
T_{n,q}^{\ast }(t;x) =\frac{2qn}{(n+\beta)[2]_q}x+\frac{n}{(n+\beta)[2]_q[n]_q}+\frac{\alpha}{n+\beta}.
\end{eqnarray*}

\item If $f(t)=t^{2}$ then from \eqref{snss1}, \eqref{snmn4}, \eqref{snb24}
and \eqref{snb25} we get (3). \newline

\item If $f(t)=t^{3}$ then from \eqref{snss1}, \eqref{snmn242}, \eqref{snb24}%
, \eqref{snb25}, \eqref{snub26} and \eqref{snb26} we get (4)
\item If $f(t)=t^{4}$ then from \eqref{snss1}, \eqref{snmn24}, \eqref{snb24}%
, \eqref{snb25}, \eqref{snub26} and \eqref{snb27} we have (4)
\end{enumerate}
\end{proof}

\begin{lemma}
\label{snlm2} Let the operators $T_{n,q}^{\ast }(.~;~.)$ be given by %
\eqref{snss1}. Then

\begin{enumerate}
\item \label{snlm22} $T_{n,q}^*(t-x;x)=\left(\frac{2qn}{(n+\beta)[2]_q}-1\right)x+\frac{n}{(n+\beta)[2]_q[n]_q}+\frac{\alpha}{n+\beta}$

\item \label{snlm23} $T_{n,q}^*((t-x)^2;x)\leq \frac{n}{(n+\beta)^2[n]_q}(\frac{n}{[3]_q[n]_q}+\frac{2\alpha}{[2]_q})+\frac{\alpha^2}{(n+\beta)^2}+\{\frac{n^2}{(n+\beta)^2}\frac{3}{[3]_q[n]_q}
    (1+[1+2\mu]_q)+\frac{2n}{(n+\beta)[2]_q}(2\alpha-\frac{1}{[n]_q})-\frac{2\alpha}{n+\beta}\}x+\{\frac{n}{n+\beta}(\frac{3n}{(n+\beta)[3]_q}
    -\frac{4n}{n+\beta)[2]_q})+1\}x^2$,

\item \label{snlm24} $T_{n,q}^*((t-x)^4;x)\leq \frac{n}{(n+\beta)^4[n]_q}(\frac{n^3}{[5]_q[n]_q^3}+\frac{4n^2\alpha}{[4]_q[n]_q^2}+\frac{}{[3]_q[n]_q}+\frac{}{[2]_q})+\frac{\alpha^4}{(n+\beta)^4}\\
    \{\frac{n^2}{(n+\beta)^4}(\frac{5n^2}{[5]_q[n]_q^3}+\frac{16n\alpha}{[4]_q[n]_q^2}+\frac{18\alpha^2}{[3]_q[n]_q}+\frac{8n\alpha}{[2]_q})
    +\frac{2n^2}{(n+\beta)^4}(\frac{5n^2}{[5]_q[n]_q^2}+\frac{12n\alpha}{[4]_q[n]_q}+\frac{9\alpha^2}{[3]_q})\\
    \times \frac{[1+2\mu]_q}{[n]_q}+\frac{2n^3}{(n+\beta)^4}(\frac{5n}{[5]_q[n]_q}+\frac{8\alpha}{[n]_q})\frac{[1+2\mu]_q^2}{[n]_q^2}
    +\frac{5n^4}{(n+\beta)^4[5]_q}\frac{[1+2\mu]_q^3}{[n]_q^3}-\frac{4n}{n+\beta)^3[3]_q}(\frac{n^2}{[4]_q[n]_q^2}+\frac{3n\alpha}{[3]_q[n]_q}
    +\frac{3\alpha^2}{[2]_q})-\frac{4\alpha^3}{(n+\beta)^3}\}x+\{\frac{2n^2}{(n+\beta)^4}(\frac{5n^2}{[5]_q[n]_q^2}+\frac{12n\alpha}{[4]_q[n]_q}
    +\frac{9\alpha^2}{[3]_q})+\frac{6n^3}{(n+\beta)^4}\\
    (\frac{5n}{[5]_q[n]_q}+\frac{8\alpha}{[n]_q})\frac{[1+2\mu]_q}{[n]_q}+\frac{35n^4}{(n+\beta)^4[5]_q}\frac{[1+2\mu]_q^2}{[n]_q^2}
    -\frac{4}{(n+\beta)^3}(\frac{4n^3}{[4]_q[n]_q^2}(1+[1+2\mu]_q^2)\\
    +\frac{9n^2\alpha}{[3]_q[n]_q}(1+[1+2\mu]_q)+6n(\frac{\alpha^2}{[2]_q}+\frac{n^2}{[4]_q}\frac{[1+2\mu]_q^2}{[n]_q^2}))
    +6(\frac{n^2}{(n+\beta)^2}\frac{1}{[3]_q[n]_q^2}\\
    +\frac{2n\alpha}{(n+\beta)^2[2]_q[n]_q}+\frac{\alpha^2}{(n+\beta)^2})\}x^2+\{\frac{2n^3}{(n+\beta)^4}(\frac{5n}{[5]_q[n]_q}
    +\frac{8\alpha}{[n]_q})+\frac{5n^4}{(n+\beta)^4[5]_q}\frac{[1+2\mu]_q}{[n]_q}\\
    -\frac{4}{(n+\beta)^3}(\frac{6n^3}{[4]_q[n]_q}(1+2[1+2\mu]_q)+\frac{9n^2\alpha}{[3]_q})+6(\frac{n^2}{(n+\beta)^2}\frac{3}{[3]_q[n]_q}
    +\frac{2n\alpha}{(n+\beta)^2}\frac{2}{[2]_q})\\
    +\frac{6n^2}{(n+\beta)^2}\frac{3}{[3]_q}\frac{[1+2\mu]_q}{[n]_q}-4(\frac{n}{(n+\beta)[2]_q[n]_q}+\frac{\alpha}{n+\beta})\}x^3
    +\{\frac{5n^4}{(n+\beta)^4[5]_q}-\frac{16n^3}{(n+\beta)^3[4]_q}\\
    \frac{6n^2}{(n+\beta)^2}\frac{3}{[3]_q}-\frac{8n}{(n+\beta)[2]_q}+1\}x^4$.
\end{enumerate}
\end{lemma}



\section{Main results}

We obtain the Korovkin's type approximation properties for our operators
defined by \eqref{snss1}.\newline

Let $C_{B}(\mathbb{R^{+}})$ be the set of all bounded and continuous
functions on $\mathbb{R^{+}}=[0,\infty )$, which is linear normed space with
\begin{equation*}
\parallel f\parallel _{C_{B}}=\sup_{x\geq 0}\mid f(x)\mid .
\end{equation*}%
Let
\begin{equation*}
H:=\{f:x\in \lbrack 0,\infty ),\frac{f(x)}{1+x^{2}}~~~\mbox{is}~~~%
\mbox{convergent}~~~\mbox{as}~~~x\rightarrow \infty \}.
\end{equation*}


\parindent=8mmIn order to obtain the convergence results for the operators $%
T_{n,q}^{\ast }(.~;~.)$, we take $q=q_{n}$ where $q_{n}\in (0,1)$ such that
\begin{equation}
\lim_{n}q_{n}\rightarrow 1,~~~~~~\lim_{n}q_{n}^{n}\rightarrow a
\label{snnas5}
\end{equation}

\begin{theorem}
\label{snth1} Let $q=q_{n}$ satisfying \eqref{snnas5}, for $0<q_{n}<1$ and
if $T_{n,q_{n}}^{\ast }(.~;~.)$ be the operators given by \eqref{snss1}.
Then for any function $f\in C[0,\infty )\cap H$,
\begin{equation*}
\lim_{n\rightarrow \infty }T_{n,q_{n}}^{\ast }(f;x)=f(x)
\end{equation*}%
is uniformly on each compact subset of $[0,\infty )$.
\end{theorem}

\begin{proof}
The proof is based on the well known Korovkin's theorem regarding the
convergence of a sequence of linear and positive operators, so it is enough
to prove the conditions
\begin{equation*}
\lim_{n\rightarrow \infty }{T}_{n,q_{n}}^{\ast
}((t^{j};x)=x^{j},~~~j=0,1,2,~~~\{\mbox{as}~n\rightarrow \infty \}
\end{equation*}%
uniformly on $[0,1]$.\newline
Clearly from \eqref{snnas5} and $\frac{1}{[n]_{q_{n}}}\rightarrow
0~~(n\rightarrow \infty )$ we have

\begin{equation*}
\lim_{n \to \infty}{T}_{n,q_n}^*(t;x)=x,~~~\lim_{n \to \infty}{T}%
_{n,q_n}^*(t^2;x)=x^2.
\end{equation*}
Which completeS the proof.
\end{proof}


We recall the weighted spaces of the functions on $\mathbb{R}^{+}$, which
are defined as follows:
\begin{eqnarray*}
P_{\rho }(\mathbb{R}^{+}) &=&\left\{ f:\mid f(x)\mid \leq M_{f}\rho
(x)\right\} , \\
Q_{\rho }(\mathbb{R}^{+}) &=&\left\{ f:f\in P_{\rho }(\mathbb{R}^{+})\cap
C[0,\infty )\right\} , \\
Q_{\rho }^{k}(\mathbb{R}^{+}) &=&\left\{ f:f\in Q_{\rho }(\mathbb{R}^{+})~~~%
\mbox{and}~~~\lim_{x\rightarrow \infty }\frac{f(x)}{\rho (x)}=k(k~~~\mbox{is}%
~~~\mbox{a}~~~\mbox{constant})\right\} ,
\end{eqnarray*}%
where $\rho (x)=1+x^{2}$ is a weight function and $M_{f}$ is a constant
depending only on $f$. Note that $Q_{\rho }(\mathbb{R}^{+})$ is a normed
space with the norm $\parallel f\parallel _{\rho }=\sup_{x\geq 0}\frac{\mid
f(x)\mid }{\rho (x)}$.

\begin{theorem}
\label{snth2} Let $q=q_n$ satisfying \eqref{snnas5}, for $0<q_n< 1$ and if $%
T_{n,q_n}^*(.~;~.)$ be the operators given by \eqref{snss1}. Then for any
function $f \in Q^k_\rho(\mathbb{R}^+)$ we have
\begin{equation*}
\lim_{n\to \infty} \parallel T_{n,q_n}^*(f;x)-f \parallel_\rho=0.
\end{equation*}
\end{theorem}

\begin{proof}
From Lemma \ref{snlm1}, the first condition of \eqref{snlm11} is fulfilled
for $\tau=0$. Now for $\tau=1,2$ it is easy to see that from (\ref{snlm12}),
(\ref{snlm13}) of Lemma \ref{snlm1} by using \eqref{snnas5} \newline
\begin{equation*}
\parallel T_{n,q_n}^* \left( t) ^{\tau};x\right) -x ^{\tau }\parallel
_{\rho} =0.
\end{equation*}%
This completes the proof.
\end{proof}

\section{\textbf{Rate of Convergence}}

Here we calculate the rate of convergence of operators \eqref{snss1} by
means of modulus of continuity and Lipschitz type maximal functions.

Let $f\in C[0,\infty ]$. The modulus of continuity of $f$ denoted by $\omega
(f,\delta )$ gives the maximum oscillation of $f$ in any interval of length
not exceeding $\delta >0$ and it is given by
\begin{equation}
\omega (f,\delta )=\sup_{\mid y-x\mid \leq \delta }\mid f(y)-f(x)\mid
,~~~x,y\in \lbrack 0,\infty ).  \label{snson1}
\end{equation}%
It is known that $\lim_{\delta \rightarrow 0+}\omega (f,\delta )=0$ for $%
f\in C[0,\infty )$ and for any $\delta >0$ one has
\begin{equation}
\mid f(y)-f(x)\mid \leq \left( \frac{\mid y-x\mid }{\delta }+1\right) \omega
(f,\delta ).  \label{snson2}
\end{equation}

\begin{theorem}
Let $T_{n,q}^*(.~;~.)$ be the operators defined by \eqref{snss1}. Then for $%
f\in \tilde{C}[0,\infty),~~x\geq 0, ~~0<q< 1$ we have
\begin{equation*}
T_{n,q}^*(f;x)-f(x)\leq \left\{ 1+\sqrt{\phi_n(x)}\right\} \omega\left(f;\frac{1}{\sqrt{
[n]_{q}}}\right),
\end{equation*}
where
\begin{eqnarray*}
\phi_n(x)=\frac{n}{(n+\beta)^2[n]_q}(\frac{n}{[3]_q[n]_q}+\frac{2\alpha}{[2]_q})+\frac{\alpha^2}{(n+\beta)^2}+\{\frac{n^2}{(n+\beta)^2}\frac{3}{[3]_q[n]_q}
    (1+[1+2\mu]_q)\\
    +\frac{2n}{(n+\beta)[2]_q}(2\alpha-\frac{1}{[n]_q})-\frac{2\alpha}{n+\beta}\}x+\{\frac{n}{n+\beta}(\frac{3n}{(n+\beta)[3]_q}
    -\frac{4n}{n+\beta)[2]_q})+1\}x^2
\end{eqnarray*}
and $\tilde{C}[0,\infty)$ is the space of uniformly continuous functions
on $\mathbb{R}^+$ and $\omega(f,\delta)$ is the modulus of continuity of the
function $f \in \tilde{C}[0,\infty)$ defined in \eqref{snson1}.
\end{theorem}

\begin{proof}
We prove it by using \eqref{snson1}, \eqref{snson2} and Cauchy-Schwarz
inequality.\newline
$\mid T_{n,q}^{\ast }(f;x)-f(x)\mid $
\begin{eqnarray*}
&\leq &\frac{[n]_{q}}{e_{\mu ,q}([n]_{q}x)}\sum_{k=0}^{\infty }\frac{%
([n]_{q}x)^{k}}{\gamma _{\mu ,q}(k)}\int_{\frac{q[k+2\mu \theta _{k}]_{q}}{%
[n]_{q}}}^{\frac{[k+1+2\mu \theta _{k}]_{q}}{[n]_{q}}}\mid f(t)-f(x)\mid
d_{q}(t) \\
&\leq &\frac{[n]_{q}}{e_{\mu ,q}([n]_{q}x)}\sum_{k=0}^{\infty }\frac{%
([n]_{q}x)^{k}}{\gamma _{\mu ,q}(k)}\int_{\frac{q[k+2\mu \theta _{k}]_{q}}{%
[n]_{q}}}^{\frac{[k+1+2\mu \theta _{k}]_{q}}{[n]_{q}}}\left( 1+\frac{1}{%
\delta }\mid t-x\mid \right) d_{q}(t)\omega (f;\delta ) \\
&=&\left\{ 1+\frac{1}{\delta }\left( \frac{[n]_{q}}{e_{\mu ,q}([n]_{q}x)}%
\sum_{k=0}^{\infty }\frac{([n]_{q}x)^{k}}{\gamma _{\mu ,q}(k)}\int_{\frac{%
q[k+2\mu \theta _{k}]_{q}}{[n]_{q}}}^{\frac{[k+1+2\mu \theta _{k}]_{q}}{%
[n]_{q}}}\mid t-x\mid d_{q}(t)\right) \right\} \omega (f;\delta ) \\
&\leq &\left\{ 1+\frac{1}{\delta }\left( \frac{[n]_{q}}{e_{\mu ,q}([n]_{q}x)}%
\sum_{k=0}^{\infty }\frac{([n]_{q}x)^{k}}{\gamma _{\mu ,q}(k)}\int_{\frac{%
q[k+2\mu \theta _{k}]_{q}}{[n]_{q}}}^{\frac{[k+1+2\mu \theta _{k}]_{q}}{%
[n]_{q}}}(t-x)^{2}d_{q}(t)\right) ^{\frac{1}{2}}\left( T_{n,q}^{\ast
}(1;x)\right) ^{\frac{1}{2}}\right\} \omega (f;\delta ) \\
&=&\left\{ 1+\frac{1}{\delta }\left( T_{n,q}^{\ast }(t-x)^{2};x\right) ^{%
\frac{1}{2}}\right\} \omega (f;\delta ) \\
&&
\end{eqnarray*}%
if we choose $\delta =\delta _{n}=\sqrt{\frac{1}{[n]_{q}}}$, then we get our
result.
\end{proof}

Now we give the rate of convergence of the operators ${T}_{n,q}^*(f;x) $
defined in \eqref{snss1} in terms of the elements of the usual Lipschitz
class $Lip_{M}(\nu )$.

Let $f\in C[0,\infty )$, $M>0$ and $0<\nu \leq 1$. The class $Lip_{M}(\nu )$
is defined as
\begin{equation}
Lip_{M}(\nu )=\left\{ f:\mid f(\zeta _{1})-f(\zeta _{2})\mid \leq M\mid
\zeta _{1}-\zeta _{2}\mid ^{\nu }~~~(\zeta _{1},\zeta _{2}\in \lbrack
0,\infty ))\right\}  \label{snn1}
\end{equation}

\begin{theorem}
\label{snsn1} Let $T_{n,q}^*(.~;~.)$ be the operator defined in \eqref{snss1}%
. Then for each $f\in Lip_{M}(\nu ),~~(M>0,~~~0<\nu \leq 1)$ satisfying %
\eqref{snn1} we have
\begin{equation*}
\mid T_{n,q}^*(f;x)-f(x)\mid \leq M \left(\lambda_{n}(x)\right)^{\frac{\nu}{2%
}}
\end{equation*}
where $\lambda_{n}(x)=T_{n,q}^*\left((t-x)^2;x\right)$.
\end{theorem}

\begin{proof}
We prove it by using \eqref{snn1} and H\"{o}lder inequality.
\begin{eqnarray*}
\mid T_{n,q}^{\ast }(f;x)-f(x)\mid &\leq &\mid T_{n,q}^{\ast
}(f(t)-f(x);x)\mid \\
&\leq &T_{n,q}^{\ast }\left( \mid f(t)-f(x)\mid ;x\right) \\
&\leq &\mid MT_{n,q}^{\ast }\left( \mid t-x\mid ^{\nu };x\right) .
\end{eqnarray*}%
Therefore\newline

$\mid T_{n,q}^*(f;x)-f(x) \mid$
\begin{eqnarray*}
&\leq & M \frac{[n]_q}{e_{\mu,q}([n]_qx)}\sum_{k=0}^\infty \frac{([n]_qx)^{k}%
}{\gamma_{\mu,q}(k)} \int_{\frac{q[k+2\mu \theta_k]_q}{[n]_q}}^{\frac{%
[k+1+2\mu \theta_k]_q}{[n]_q}}\mid t-x \mid^\nu d_q(t) \\
& \leq & M \frac{[n]_q}{e_{\mu,q}([n]_qx)}\sum_{k=0}^\infty \left(\frac{%
([n]_qx)^{k}}{\gamma_{\mu,q}(k)}\right)^{\frac{2-\nu}{2}} \\
& \times & \left(\frac{([n]_qx)^{k}}{\gamma_{\mu,q}(k)}\right)^{\frac{\nu}{2}%
} \int_{\frac{q[k+2\mu \theta_k]_q}{[n]_q}}^{\frac{[k+1+2\mu \theta_k]_q}{%
[n]_q}}\mid t-x \mid^\nu d_q(t) \\
& \leq & M \left(\frac{[n]_q}{\left(e_{\mu,q}([n]_qx)\right)}%
\sum_{k=0}^\infty \frac{([n]_qx)^{k}}{\gamma_{\mu,q}(k)}\int_{\frac{q[k+2\mu
\theta_k]_q}{[n]_q}}^{\frac{[k+1+2\mu \theta_k]_q}{[n]_q}} d_q(t)\right)^{%
\frac{2-\nu}{2}} \\
& \times & \left(\frac{[n]_q}{\left(e_{\mu,q}([n]_qx)\right)}%
\sum_{k=0}^\infty \frac{([n]_qx)^{k}}{\gamma_{\mu,q}(k)} \int_{\frac{q[k+2\mu
\theta_k]_q}{[n]_q}}^{\frac{[k+1+2\mu \theta_k]_q}{[n]_q}}\mid t-x \mid^2
d_q(t) \right)^{\frac{\nu}{2}} \\
& = & M \left(T_{n,q}^*(t-x)^2;x\right)^{\frac{\nu}{2}}.
\end{eqnarray*}
Which completes the proof.
\end{proof}

Let $C_{B}[0,\infty )$ denote the space of all bounded and continuous
functions on $\mathbb{R}^{+}=[0,\infty )$ and
\begin{equation}
C_{B}^{2}(\mathbb{R}^{+})=\{g\in C_{B}(\mathbb{R}^{+}):g^{\prime },g^{\prime
\prime }\in C_{B}(\mathbb{R}^{+})\},  \label{snt2}
\end{equation}%
with the norm
\begin{equation}
\parallel g\parallel _{C_{B}^{2}(\mathbb{R}^{+})}=\parallel g\parallel
_{C_{B}(\mathbb{R}^{+})}+\parallel g^{\prime }\parallel _{C_{B}(\mathbb{R}%
^{+})}+\parallel g^{\prime \prime }\parallel _{C_{B}(\mathbb{R}^{+})},
\label{snt1}
\end{equation}%
also
\begin{equation}
\parallel g\parallel _{C_{B}(\mathbb{R}^{+})}=\sup_{x\in \mathbb{R}^{+}}\mid
g(x)\mid .  \label{snt3}
\end{equation}

\begin{theorem}
\label{snsn2} Let $T_{n,q}^*(.~;~.)$ be the operator defined in \eqref{snss1}%
. Then for any $g \in C_B^2(\mathbb{R}^+)$ we have
\begin{equation*}
\mid T_{n,q}^*(f;x)-f(x)\mid \leq (\frac{2qn}{(n+\beta)[2]_q}-1)x+\frac{n}{(n+\beta)[2]_q[n]_q}+\frac{\alpha}{n+\beta}+\frac{\lambda_{n}(x)}{2}) \parallel g
\parallel_{C_B^2(\mathbb{R}^+)}
\end{equation*}
where $\lambda_{n}(x)$ is given in Theorem \ref{snsn1}.
\end{theorem}

\begin{proof}
Let $g\in C_{B}^{2}(\mathbb{R}^{+})$, then by using the generalized mean
value theorem in the Taylor series expansion we have
\begin{equation*}
g(t)=g(x)+g^{\prime }(x)(t-x)+g^{\prime \prime }(\psi )\frac{(t-x)^{2}}{2}%
,~~~\psi \in (x,t).
\end{equation*}%
By applying linearity property on $T_{n,q}^{\ast },$ we have
\begin{equation*}
T_{n,q}^{\ast }(g,x)-g(x)=g^{\prime }(x)T_{n,q}^{\ast }\left( (t-x);x\right)
+\frac{g^{\prime \prime }(\psi )}{2}T_{n,q}^{\ast }\left( (t-x)^{2};x\right),
\end{equation*}%
which implies \newline
$\mid T_{n,q}^{\ast }(g;x)-g(x)\mid $
\begin{eqnarray*}
&\leq &\left(\frac{2qn}{(n+\beta)[2]_q}-1\right)x+\frac{n}{(n+\beta)[2]_q[n]_q}+\frac{\alpha}{n+\beta} \parallel g^{\prime }\parallel _{C_{B}(\mathbb{R}^{+})} \\
&+&\frac{n}{(n+\beta)^2[n]_q}(\frac{n}{[3]_q[n]_q}+\frac{2\alpha}{[2]_q})+\frac{\alpha^2}{(n+\beta)^2}+\{\frac{n^2}{(n+\beta)^2}\frac{3}{[3]_q[n]_q}
    \times(1+[1+2\mu]_q)\\
    &+&\frac{2n}{(n+\beta)[2]_q}(2\alpha-\frac{1}{[n]_q})-\frac{2\alpha}{n+\beta}\}x+\{\frac{n}{n+\beta}(\frac{3n}{(n+\beta)[3]_q}
    -\frac{4n}{n+\beta)[2]_q})+1\}x^2
    \times\frac{\parallel g^{\prime \prime
}\parallel _{C_{B}(\mathbb{R}^{+})}}{2}.
\end{eqnarray*}%
On using \eqref{snt1}~~~$\parallel g^{\prime }\parallel
_{C_{B}[0,\infty )}\leq \parallel g\parallel _{C_{B}^{2}[0,\infty )}$
completes the proof from \ref{snlm23} of Lemma \ref{snlm2}.
\end{proof}

The Peetre's $K$-functional is defined by
\begin{equation}
K_{2}(f,\delta )=\inf_{C_{B}^{2}(\mathbb{R}^{+})}\left\{ \left( \parallel
f-g\parallel _{C_{B}(\mathbb{R}^{+})}+\delta \parallel g^{\prime \prime
}\parallel _{C_{B}^{2}(\mathbb{R}^{+})}\right) :g\in \mathcal{W}^{2}\right\}
,  \label{snzr1}
\end{equation}%
where
\begin{equation}
\mathcal{W}^{2}=\left\{ g\in C_{B}(\mathbb{R}^{+}):g^{\prime },g^{\prime
\prime }\in C_{B}(\mathbb{R}^{+})\right\} .  \label{snzr2}
\end{equation}%
There exits a positive constant $C>0$ such that $K_{2}(f,\delta )\leq
C\omega _{2}(f,\delta ^{\frac{1}{2}}),~~\delta >0$, where the second order
modulus of continuity is given by
\begin{equation}
\omega _{2}(f,\delta ^{\frac{1}{2}})=\sup_{0<h<\delta ^{\frac{1}{2}%
}}\sup_{x\in \mathbb{R}^{+}}\mid f(x+2h)-2f(x+h)+f(x)\mid .  \label{snzr3}
\end{equation}

\begin{theorem}
\label{snsn3} Let $T_{n,q}^{\ast }(.~;~.)$ be the operator defined in %
\eqref{snss1} and $C_{B}[0,\infty )$ be the space of all bounded and
continuous functions on $\mathbb{R}^{+}$. Then for $x\in \mathbb{R}^{+},$ $%
f\in C_{B}(\mathbb{R}^{+})$ we have\newline
$\mid T_{n,q}^{\ast }(f;x)-f(x)\mid $\newline
$\leq 2M\left\{ \omega _{2}\left( f;\sqrt{\frac{(\frac{4qn}{(n+\beta)[2]_q}-2)x
+\left( \frac{2n}{(n+\beta)[2]_q[n]_q}+\frac{2\alpha}{n+\beta}\right)+\lambda _{n}(x)}{4}}\right)\newline
+\min\left( 1,\frac{(\frac{4qn}{(n+\beta)[2]_{q}}-2)x+\left( \frac{2n}{(n+\beta)[2]_q[n]_q}+\frac{2\alpha}{n+\beta}\right)
+\lambda _{n}(x)}{4}\right) \\
\parallel f\parallel _{C_{B}(\mathbb{R}%
^{+})}\right\} $,\newline
where $M$ is a positive constant, $\lambda _{n}(x)$ is given in Theorem \ref%
{snsn1} and $\omega _{2}(f;\delta )$ is the second order modulus of
continuity of the function $f$ defined in \eqref{snzr3}.
\end{theorem}

\begin{proof}
We prove this by using the Theorem \eqref{snsn2}
\begin{eqnarray*}
\mid T_{n,q}^*(f;x)-f(x)\mid &\leq & \mid T_{n,q}^*(f-g;x)\mid+\mid
T_{n,q}^*(g;x)-g(x)\mid+\mid f(x)-g(x)\mid \\
&\leq & 2 \parallel f-g \parallel_{C_B(\mathbb{R}^+)}+ \frac{\lambda_n(x)}{2}%
\parallel g \parallel_{C_B^2(\mathbb{R}^+)} \\
&+ &\left(\frac{2qn}{(n+\beta)[2]_q}-1\right)x+\frac{n}{(n+\beta)[2]_q[n]_q}+\frac{\alpha}{n+\beta}%
\parallel g \parallel_{C_B(\mathbb{R}^+)}
\end{eqnarray*}
From \eqref{snt1} clearly we have ~~~$\parallel g
\parallel_{C_B[0,\infty)}\leq \parallel g \parallel_{C_B^2[0,\infty)}$.%
\newline
Therefore,
\begin{equation*}
\mid T_{n,q}^*(f;x)-f(x)\mid \leq 2 \left(\parallel f-g \parallel_{C_B(%
\mathbb{R}^+)}+\frac{(\frac{4qn}{(n+\beta)[2]_q}-2)x+\frac{2n}{(n+\beta)[2]_q[n]_q}+\frac{2\alpha}{n+\beta}+
\lambda_n(x)}{4}\parallel g \parallel_{C_B^2(\mathbb{R}^+)}\right),
\end{equation*}
where $\lambda_{n}(x)$ is given in Theorem \ref{snsn1}.\newline

By taking infimum over all $g\in C_{B}^{2}(\mathbb{R}^{+})$ and by using %
\eqref{snzr1}, we get
\begin{equation*}
\mid T_{n,q}^{\ast }(f;x)-f(x)\mid \leq 2K_{2}\left( f;\frac{(\frac{4qn}{(n+\beta)[2]_q}-2)x+\frac{2n}{(n+\beta)[2]_q[n]_q}+\frac{2\alpha}{n+\beta}+\lambda _{n}(x)}{4}%
\right)
\end{equation*}%
Now for an absolute constant $D>0$ in \cite{scc1} we use the relation
\begin{equation*}
K_{2}(f;\delta )\leq D\{\omega _{2}(f;\sqrt{\delta })+\min (1,\delta
)\parallel f\parallel \}.
\end{equation*}%
This complete the proof.
\end{proof}

Atakut and Ispir \cite{satk} introduced the weighted modulus of continuity
and defined as, for an arbitrary $f\in Q_{\rho }^{k}(\mathbb{R}^{+})$
\begin{equation}
\Omega (f,\delta )=\sup_{x\in \lbrack 0,\infty ),\mid h\mid \leq \delta }%
\frac{\mid f(x+h)-f(x)\mid }{(1+h^{2})(1+x^{2})}.  \label{snd}
\end{equation}%
The two main properties of this modulus of continuity are $\lim_{\delta
\rightarrow 0}\Omega (f,\delta )\rightarrow 0$ and
\begin{equation}
\mid f(t)-f(x)\mid \leq 2\left( 1+\frac{\mid t-x\mid }{\delta }\right)
(1+\delta ^{2})(1+x^{2})(1+(t-x)^{2})\Omega (f,\delta ),  \label{smd}
\end{equation}%
where $f\in Q_{\rho }^{k}(\mathbb{R}^{+})$ and $t,x\in \lbrack 0,\infty )$.

\begin{theorem}
Let $T_{n,q}^*(.~;~.)$ be the operators defined by \eqref{snss1}. Then for $%
f \in Q^k_\rho(\mathbb{R}^+),~~0<q< 1$ and $x\geq 0 $ we have \newline
\begin{equation*}
\sup_{x \in [0,\infty)}\frac{\mid T_{n,q}^*(f;x)-f(x)\mid}{(1+x^2)} \leq
C_\mu\left( 1+\frac{1}{[n]_q}\right)\Omega(f,\frac{1}{\sqrt{[n]_q}}),
\end{equation*}
where $C_\mu$ is constant independent of $n$.
\end{theorem}

\begin{proof}
We prove it by using \eqref{snd}, \eqref{smd} and Cauchy-Schwarz inequality.%
\newline
$\mid T_{n,q}^{\ast }(f;x)-f(x)\mid $
\begin{eqnarray*}
&\leq &\frac{[n]_{q}}{e_{\mu ,q}([n]_{q}x)}\sum_{k=0}^{\infty }\frac{%
([n]_{q}x)^{k}}{\gamma _{\mu ,q}(k)}\int_{\frac{q[k+2\mu \theta _{k}]_{q}}{%
[n]_{q}}}^{\frac{[k+1+2\mu \theta _{k}]_{q}}{[n]_{q}}}\mid f(t)-f(x)\mid
d_{q}(t) \\
&\leq &2(1+\delta ^{2})(1+x^{2})\Omega (f;\delta )\frac{[n]_{q}}{e_{\mu
,q}([n]_{q}x)}\sum_{k=0}^{\infty }\frac{([n]_{q}x)^{k}}{\gamma _{\mu ,q}(k)}%
\int_{\frac{q[k+2\mu \theta _{k}]_{q}}{[n]_{q}}}^{\frac{[k+1+2\mu \theta
_{k}]_{q}}{[n]_{q}}}\left( 1+\frac{1}{\delta }\mid t-x\mid \right) \left(
1+(t-x)^{2}\right) d_{q}(t) \\
&=&2(1+\delta ^{2})(1+x^{2})\Omega (f;\delta )\frac{[n]_{q}}{e_{\mu
,q}([n]_{q}x)} \\
&\times &\bigg{\{}\sum_{k=0}^{\infty }\frac{([n]_{q}x)^{k}}{\gamma _{\mu
,q}(k)}+\sum_{k=0}^{\infty }\frac{([n]_{q}x)^{k}}{\gamma _{\mu ,q}(k)}\int_{%
\frac{q[k+2\mu \theta _{k}]_{q}}{[n]_{q}}}^{\frac{[k+1+2\mu \theta _{k}]_{q}}{%
[n]_{q}}}(t-x)^{2}d_{q}(t) \\
&+&\frac{1}{\delta }\sum_{k=0}^{\infty }\frac{([n]_{q}x)^{k}}{\gamma _{\mu
,q}(k)}\int_{\frac{q[k+2\mu \theta _{k}]_{q}}{[n]_{q}}}^{\frac{[k+1+2\mu
\theta _{k}]_{q}}{[n]_{q}}}\mid t-x\mid d_{q}(t) \\
&+&\frac{1}{\delta }\sum_{k=0}^{\infty }\frac{([n]_{q}x)^{k}}{\gamma _{\mu
,q}(k)}\int_{\frac{q[k+2\mu \theta _{k}]_{q}}{[n]_{q}}}^{\frac{[k+1+2\mu
\theta _{k}]_{q}}{[n]_{q}}}\mid t-x\mid (t-x)^{2}d_{q}(t)\bigg{\}} \\
&\leq &2(1+\delta ^{2})(1+x^{2})\Omega (f;\delta ) \\
&\times &\left( 1+T_{n,q}^{\ast }((t-x)^{2};x)+\frac{1}{\delta }\sqrt{%
T_{n,q}^{\ast }((t-x)^{2};x)}+\frac{1}{\delta }\sqrt{T_{n,q}^{\ast
}((t-x)^{2};x)T_{n,q}^{\ast }((t-x)^{4};x)}\right) \\
&&
\end{eqnarray*}%
where $T_{n,q}^{\ast }((t-x)^{2};x)$ and $T_{n,q}^{\ast }((t-x)^{4};x)$ is
defined in (\ref{snlm23}) and (\ref{snlm24}) of Lemma \ref{snlm2}.\newline
If we choose $\delta =\delta _{n}=\sqrt{\frac{1}{[n]_{q}}}$, then we get our
result.
\end{proof}


\end{document}